\documentclass[10pt]{article}
\usepackage{graphicx,color,mathptmx,latexsym,amsmath, amsthm, amscd, amsfonts, amssymb}
 
 \numberwithin{equation}{section}

\newtheorem{theorem}{Theorem}[section]

\newtheorem{lemma}[theorem]{Lemma}

\newtheorem{corollary}[theorem]{Corollary}

\newtheorem{example}[theorem]{Example}

\makeatletter
\newcommand*{\rom}[1]{\expandafter\@slowromancap\romannumeral #1@}
\makeatletter
\begin{document}
\thispagestyle{plain}

\begin{center}
{\huge  The diameter of  proper power graphs  of alternating groups}
\bigskip
\vspace*{1cm} \\

  K. Pourghobadi,  Shahrood University of Technology, \\
 email: kpghobadi@gmail.com
\vspace*{0.5cm} \\
 S.H. Jafari,  Shahrood University of Technology,  \\
 email:shjafari55@gmail.com
\end{center}


\begin{abstract}
The power graph ${\mathcal{P}}(G)$ of  finite group $G$ is a simple graph whose vertex set is $G$ and two distinct elements $ \alpha $ and  $\beta$ are adjacent if and only if one of them is a power of the other. The  proper power graph of $  G $ denoted by ${\mathcal{P}}^*(G)$  is a graph which is  obtained by deleting the identity vertex  (the identity element of $G$).  
In this paper, we improve the diameter bound of ${\mathcal{P}}^*(A_{n})$ for which ${\mathcal{P}}^*(A_n)$ is connected.
  We show that  $6 \leq  diam({\mathcal{P}}^*(A_n)) \leq  11$ for $n \geq  51$. We also describe a number of short paths in these power graphs.
\end{abstract}


\section{Introduction}	
In a graph $\Gamma$, a {\it path} is an alternating sequence of vertices and edges which begins and ends at a pair of vertices such that consecutive terms are incident. 
  When every pair of vertices are joined by a path, the graph $\Gamma$ is {\it connected}. 
     The number of edges in the path is the {\it length} of a path.  
 The length of the shortest path connecting  between two vertices $x$ and $y$ in a connected graph  is  assigned  to as its the {\it distance},
  which is denoted by    $d(x, y)$. 
The maximum distance between all pairs of vertices of $\Gamma$ is  assigned to as its  the {\it diameter} of a connected graph $\Gamma$, which is denoted by $diam(\Gamma)$.  \\ 
In  a group $G$, the order of $x$ in $G$, denoted $o(x)$, is the smallest positive integer such that $x^{o(x)}=1$, where $1$ is the identity element of $G$.
 The {\it power graph} of a group $G$ is the simple graph ${\mathcal{P}} (G)$, with vertex-set $G$ and vertices $x$ and $y$ are {\it adjacent}, denoted $x \sim y$, if and only if $x \neq y$ and either $y = x^m$ or $x = y^m$ for some positive integer $m$. The identity element $1$ of a finite group $G$ is adjacent to every other vertex $x$ since $x^{o(x)} = 1$.  The {\it proper power graph} of  $G$, denoted ${\mathcal{P}}^*(G)$, is the graph obtained from ${\mathcal{P}}(G)$ by deleting the vertex $1$. 
 While the power graph of any group is connected, the proper power graph may not be.\\
     Kelarev and Quinn \cite{a7} introduced  the power digraphs for semigroups. 
      Chakrabarty et al. \cite{c1}  defined the (undirected) power graphs of  semigroups.   
 Chattopadhyay and  Panigrahi  \cite{c2} considered   connectivity and planarity of power graphs of finite cyclic, dihedral and dicyclic groups.
    Moghaddamfar et al.  \cite{m1} studied some properties of the power
graph ${\mathcal{P}}(G)$ and the subgraph ${\mathcal{P}}^*(G)$.
  In  \cite{j1},  Jafari   investigated connectivity, diameter and clique number of proper power graphs.
  In \cite{a3},  Doostabadi and Farrokhi addressed  the connectedness and diameter of the proper power graph of a alternating group.

\begin{theorem} (\cite{a3}, Theorem 4.7.)
 Let $G = A_n$ be an alternating group $(n \geq 3)$.   If  $n$, $n - 1$,  $n - 2$, $n /2$, $(n - 1)/2$, $(n - 2)/2$ are not primes, then ${\mathcal{P}}^*(G)$ is connected and $diam( {\mathcal{P}}^*(G))\leq 22$.
 \,\label{le:1.1}
\end{theorem}
 
 In this paper, we improve the diameter bound of Theorem ~\ref{le:1.1}. 
  We show that  $6 \leq  diam({\mathcal{P}}^*(A_n)) \leq  11$ for $n >  51$. 
 We also describe a number of short paths in these power graphs.\\
 Throughout this paper, we assume that  $n >51$ is such that  
 $n$, $n - 1$,  $n - 2$, $n /2$, $(n - 1)/2$, $(n - 2)/2$ are not primes.  
 
  \section{Some of short paths}  
We know  every permutation $\alpha\in A_n$ is either a cycle or factorization (which is a product of disjoint cycles). 
 We also know  that taking an appropriate power shows every permutation in $A_n$ is adjacent to a permutation of prime order. 
We determine some of  short paths between elements of prime order the proper power graph of the alternating  group. 
\begin{lemma}
  (\cite{a3}, Lemma $2.1$)
   Let $G$ be a finite group and $x, y \in G\backslash \{1\}$ such that $xy = yx$ and
$gcd(o(x), o(y)) = 1$. Then $x \thicksim xy\thicksim y$. In particular,  $d(x, y)=2$.
\,\label{le:1}
\end{lemma}
\hspace*{-7mm}
Due to the  Lemma ~\ref{le:1}, the  distance a permutation of prime order from any permutation of its fixed points having a different prime order is $2$.
 For each $\alpha\in A_n$, the  support of $\alpha$ is defined by $S(\alpha)=\{i: \ \alpha(i)\neq i\}$ and show the complement of set $S(\alpha)$ by $S^{c}(\alpha) $.  
\begin{lemma} (\cite{a3}, Lemma 4.3.)
\label{le:whole}
  For $n \geq 10$, there is a path of length at most four in ${\mathcal{P}}^*(A_n)$ between any two $3$-cycles.
  \,\label{le:00} 
\end{lemma}
 
  \begin{lemma}
\label{le:whole}
  Let  $ \alpha, \beta\in A_n $  have respective prime orders $p$ and $q$ such that
     $ \alpha=\alpha_1 \cdots\alpha_m$ and $ \beta=\beta_1  \cdots \beta_{m'} $. If $p q\leq 9$, then $d(\alpha, \beta)\leq 6$.   
  \,\label{le:3}
\end{lemma}
\begin{proof} 
Assume that $ k=n-|S(\alpha)| $, $ k'=n-|S(\beta)| $.  	We consider three cases as follows:\\
 {\bf Case 1:} 
 Let $ p=q=3 $.
 \begin{itemize}
 \item[$(1.1)$] If  $ k, k'\geq 4 $, then there  are   $ j_1, j_2, j_3,  j_4  \notin S(\alpha )$ and  $x=(j_1, j_2) (j_3, j_4)$. 
 Since $\alpha$ and $x$ commute and $gcd(o(\alpha), o(x))=1$, Lemma ~\ref{le:1} implies that $ d(\alpha, x) =2 $.
  If $ S(x)\cap S(\beta)=\emptyset$, by Lemma ~\ref{le:1}, $ d(x, \beta)=2 $. Therefore  $ d(\alpha, \beta)\leq 4 $.  If  $ S(x)\cap S(\beta)\neq\emptyset$ and $ n-|S(\beta)\cup S(x)|\geq5 $, then 
  for distinct $ i_1, i_2, i_3, i_4, i_5\notin S(\beta)\cup S(x)  $, there is a   $5$-cycle $ y=(i_1, i_2, i_3, i_4, i_5)$.  
   Since $y$ and $x$ commute and $gcd(o(y), o(x))=1$, Lemma ~\ref{le:1} implies that $ d(y, x) =2 $.
  Likewise, $d(y, \beta)=2$.
   Therefore $ d(\alpha, \beta)\leq 6 $. 
  Otherwise, $m' >8$ and there are $\beta_{j_6}, \beta_{j_7}, \beta_{j_8}$ and $\beta_{j_9}$ such that $S(x)\subseteq S(\beta_{j_6}\cdots\beta_{j_9})\cup S^c(\beta)$.  
    Say  $\beta_{j_1}=(z_{11}, z_{12}, z_{13}), \beta_{j_2}=(z_{21}, z_{22}, z_{23}),  \cdots$, $ \beta_{j_5}=(z_{51}, z_{52}, z_{53})$, and set     
 $$\sigma=(z_{11}, z_{21}, z_{31}, z_{41}, z_{51}, z_{12}, z_{22}, z_{32 }, z_{42}, z_{52}, z_{13}, z_{23}, z_{33}, z_{43}, z_{53}).$$
  Then $\beta\sim \sigma   \beta_{j_{6}}  ^l\cdots \beta_{j_{m'}} ^l\sim\sigma^3 $,   in which $({\beta_j ^l})^5=\beta_j $.
      By Lemma ~\ref{le:1},  $ d( \sigma^3, x) =2 $.   Therefore  $ d(\alpha, \beta)\leq 6  $. 
\item[$(1.2)$]  If $ k\geq 4 $  and $ k' <4 $,  then  $m'> 16$. Since $ k\geq 4 $,  then for distinct $ j_1, j_2, j_3,  j_4  \notin S(\alpha )$,  there is  $x=(j_1, j_2) (j_3, j_4)$ with  $d(\alpha, x) = 2$.
 Since $m'>16$, as above there is a $15$-cycle $\sigma$ such that $S(x)\cap S(\sigma)=\emptyset$ and  $d(\beta, \sigma^3)=2$.
By Lemma ~\ref{le:1},  $ d( \sigma^3, x) =2 $.   Therefore  $ d(\alpha, \beta)\leq 6 $.
\item[$(1.3)$]  If $ k'\geq 4 $  and $ k  <4 $, then  the proof is similar to    $(1.2)$.
\item[$(1.4)$]   If $ k< 4 $  and $ k' <4 $,  then  $m'> 16$ and $m>16$. Let   $\alpha_{ 1}=(r_{11}, r_{12} , r_{13})$, $ \alpha_{ 2}=(r_{21}, r_{22}, r_{23})$, $\alpha_{ 3}=(r_{31}, r_{32}, r_{33})$, $\alpha_{ 4}=(r_{41}, r_{42}, r_{43})$, and set   $\tau=(r_{11}, r_{21}, r_{12}$, $ r_{22}, r_{13}, r_{23})  (r_{31}, r_{41}$, $ r_{32}, r_{42}, r_{33}, r_{43})$.
Then $\alpha\sim \tau  \alpha_{5}  ^l\cdots \alpha_{ {m}} ^l\sim\tau^3 $,   in which $({\alpha_i ^l})^3=\alpha_i $.
There are   $  \beta _{j_1},\cdots,  \beta _{j_{t-1}}$   and $\beta _{j_t}$ such that     $S(\tau)\subseteq S(\beta _{j_1}\cdots\beta _{j_t})\cup S^c(\beta)$. 
  Since $t\leq 12$ and  $m'>16$, there is $\sigma$ in $A_n$ such that $ d(\beta, \sigma)=2$, $o(\sigma)=5$ and $  S(\sigma)\subseteq S(\beta)-S(\beta _{j_1} \cdots \beta _{j_t})$. By Lemma ~\ref{le:1},  $ d( \sigma, \tau^3) =2 $.   Therefore  $ d(\alpha, \beta)\leq 6 $. 
 \end{itemize} 
 {\bf Case 2:}  Let $ p=q=2 $, the proof is similar to  case $ 1 $. \\
  {\bf Case 3:} Let $ p=2 $ and $ q=3 $.
  Assume that $ k \geq 3 $ and $ k'\geq 4 $.  If $ S(\alpha)\cap S(\beta)=\emptyset $, then $ d(\alpha,\beta)=2 $.  
 Suppose  $ S(\alpha)\cap S(\beta)\neq\emptyset $.
  For $n-|S(\alpha)\cup S(\beta)|\geq 5$,  there is a   $5$-cycle $y$ such that $d(\alpha, y)=d(\beta, y)=2$. Thus  $ d(\alpha, \beta)\leq 4 $.\\ For  $n-|S(\alpha)\cup S(\beta)|< 5$, since $k\geq 3$, for distinct $c_1, c_2, c_3\notin S(\alpha)$, the $3$-cycle $c=(c_1, c_2, c_3)$ is at distance $2$ from $\alpha$.
   If $m'\geq 7$, then there are $x$, $\beta_{j_1}$, $\beta_{j_2}$, $\beta_{j_3}$ and $\beta_{j_4}$ such that $S(c)\cap S(\beta_{j_1}\beta_{j_2}\beta_{j_3}\beta_{j_4}) =\emptyset$, $d(x, \beta)=2$ and  $S(x)=S(\beta_{j_1} \cdots \beta_{j_4})$, where $o(x)=2$.  
    By Lemma ~\ref{le:1},   $d(x, c)=2$.  Therefore  $ d(\alpha, \beta)\leq 6 $.\\
 If $m'<7$, then $n-|S(\beta)\cup S(c)|\geq 4$.  Hence for distinct $ j_1, j_2, j_3,  j_4  \notin S(\beta)\cup S(c)$,  there is  $x=(j_1, j_2) (j_3, j_4)$.  By Lemma ~\ref{le:1},   $d(c, x)=d(x,\beta)=2$. Therefore  $ d(\alpha, \beta)\leq 6 $. The remaining cases are similar.\\
   
\end{proof}
\begin{lemma}
\label{le:whole}
  Let  $ \alpha, \beta\in A_n $  have respective prime orders $p$ and $q$ such that
     $ \alpha=\alpha_1 \cdots \alpha_m$ and $ \beta=\beta_1  \cdots \beta_{m'} $, where $p\leq q$.   
 If   $ m'=3$, $n-|S(\beta)|<3$ and  $n-|S(\alpha)|<3 $, then $d(\alpha, \beta)\leq 10$.
     Otherwise, $ d(\alpha, \beta)\leq 8$.
  \,\label{le:8}
\end{lemma}
\begin{proof} 
  Assume that $k=n-|S(\alpha)|$ and $k'=n-|S(\beta)|$. When  $k<3$, since ${\mathcal{P}}^*(A_n)$ is connected and $o( \alpha )$ is a prime,  $m\geq 3$. Likewise,  when $k'<3$,  $m'\geq 3$.\\
For $ k'<3$ and $m'=3$, 
 let  $\beta_{ 1}=(z_{11}, \cdots, z_{1q})$, $ \beta_{ 2}=(z_{21}, \cdots, z_{2q})$,     $ \beta_{ 3}=(z_{31}, \cdots, z_{3q})$, and set   $y=(z_{11}, z_{21}, z_{31}, z_{12}, z_{22}, z_{32}, \cdots$, $ z_{1q }, z_{2q}, z_{3q} )$. Then $ \beta\thicksim y \thicksim y^q$ and let $y^q=y_1 \cdots y_q $ be a product of  $3$-cycles.\\
  If   $k<3$, we have  $m\geq 3$. As above there is a $3p$-cycle $x$ such that   $ \alpha\thicksim x \thicksim x^p $
  and let $x^p=x_1 \cdots x_p $  be a product of  $3$-cycles. 
   By Lemma ~\ref{le:3}, $d(x^p, y^q)\leq 6$, thus  $  d(\alpha, \beta)\leq 10 $.\\
Let  $k\geq 3$ and $p\neq 3$.  There is a $3$-cycle $c $  with $d(\alpha, c)=2$.
 Since $q\geq 13$, there are  $\tau$, $y_{i_1}$, $y_{i_2}$, $y_{i_3}$ and $y_{i_4}$ such that $S(c)\cap S(y_{i_1}y_{i_2}y_{i_3} y_{i_4})=\emptyset$, $d(y^q, \tau)=2$ and  $S(\tau)=  S(y_{i_1}y_{i_2}y_{i_3} y_{i_4})$, where $o(\tau)=2$. 
By Lemma ~\ref{le:1},  $d(c, \tau)=2$.
 Consequently $d(y^q, c)\leq 4$ and $  d(\alpha, \beta)\leq 8 $.\\
Let  $k\geq 3$ and $p= 3$.   Say   $\alpha_{ 1}=(r_{11}, r_{12} , r_{13})$, $ \alpha_{ 2}=(r_{21}, r_{22}, r_{23})$, and set   $\sigma=(r_{11}, r_{21}, r_{12}, r_{22}, r_{13}, r_{23})$.  Then  for distinct $u, v\notin S(\alpha)$, $\alpha\sim (u, v)\sigma {\alpha_3}^l\cdots {\alpha_m}^l\sim (u, v)\sigma ^3$, in which $({{\alpha_j}^l})^2={\alpha_j} $. 
Since $q\geq 13$, there are  $\tau$, $y_{i_1}$, $\cdots$, $y_{i_4}$ and $y_{i_5}$ such that $S((u, v)\sigma)\cap S(y_{i_1}\cdots y_{i_5}  )=\emptyset$, $d(y^q, \tau)=2$ and $S(\tau)=  S(y_{i_1} \cdots y_{i_5})$, where $o(\tau)=5$.
By Lemma ~\ref{le:1},  $d((u, v)\sigma^3, \tau)=2$.
 Consequently $d(y^q, (u, v)\sigma^3)\leq 4$ and $  d(\alpha, \beta)\leq 8 $.\\
For the rest of the  proof, we consider four cases as follows:\\  
{\bf Case $ (1)$:} Let $ p=q $ and $ p\geq 5$.\\
We know if  $k<3$, then  for $p<11$, $m\geq7$,  for   $p=11$, $m\geq5$ and for  $p\geq 13$, $m\geq4$.
   
\begin{itemize}
\item[$(1.1)$]   
Let $k\geq 3$ and $k'\geq 3$. Then there are cycles $c$ and $c'$ of length $3$ in $A_n$ such that $ d(\alpha, c)=  d(\beta, c') =2 $. 
    By Lemma ~\ref{le:00}, $ d(c, c')\leq 4 $, thus $ d(\alpha, \beta)\leq 8 $. 
\item[$(1.2)$]  Let $k<3$,  $k'<3$. Then   $m'>3$. 
Assume that  $\alpha'=\alpha_4\cdots \alpha_{m }$, $\gamma_1=\beta_4\beta_5\cdots\beta_{m'}$, $\gamma_2=\beta_1\beta_5\cdots\beta_{m'}$,  $\gamma_3=\beta_2\beta_5\cdots\beta_{m'}$ and $\gamma_4=\beta_3\beta_5\cdots\beta_{m'}$. Since $|(S(\alpha')\cap S(\gamma_1)) \cup (S(\alpha')\cap S(\gamma_2))\cup (S(\alpha')\cap S(\gamma_3)) \cup (S(\alpha')\cap S(\gamma_4))|\geq (m-3)p\geq 13$,  there is $\gamma_j$  such that $|(S(\alpha')\cap S(\gamma_j))|\geq 4$. Hence there are $\tau$ and $ \sigma$  in $A_n$ such that $d(\alpha,\tau^p)= d(\beta, \sigma^q)=2$, $S(\tau)=S(\alpha_1\alpha_2\alpha_3)$, $ S(\sigma)=S(\beta)-S(\gamma_j)$,  $o(\tau)=3p$ and $o(\sigma)=3q$.
 Since $n-|S(\tau )\cup S(\sigma )|\geq 4$,  for distinct $ j_1, j_2, j_3,  j_4  \notin S(\tau )\cup S(\sigma )$,  there is  $x=(j_1, j_2) (j_3, j_4)$.  By Lemma ~\ref{le:1},   $d(x, \tau^p )=d(x, \sigma^q)=2$.  Consequently    $d(\alpha,\beta)\leq 8$.  
\item[$(1.3)$]  Let $k<3$ and  $k'\geq 3$.  
Since $k'\geq 3$,  there is  a $3$-cycle $c $   with $d(\beta,c)=2$. 
Since $k<3$, $m\geq 3$. As proved above there is a $3p$-cycle $ \tau$ such that   $  d(\alpha, \tau^p)= 2$  and $S(c)\cap S(\tau)\geq 1$. 
 Hence $n-|S(\tau )\cup S(c)| \geq 4$ and there is  $x $  in $A_n$ such that    $d(x, \tau^p )=d(x, c)=2$ and $o(x)=2$.
Consequently    $d(\tau^p , c)\leq 4$ and $d(\alpha,\beta)\leq 8$.
 \item[$(1.4)$] Let $k\geq 3$ and  $k'< 3$.  The proof is similar to   $(1.3)$.
\end{itemize} 
 
 The proof of the following cases is similar to  the previous case.\\
{\bf Case $ (2)$:}
 Let $ p\neq q $, $ p\geq 5$ and $ q\geq 5 $.  \\
{\bf Case $ (3)$:} Let $ q\geq 5 $ and $ p=3 $.  \\
 {\bf Case $ (4)$:} Let $ q\geq 5$  and $ p=2 $.   
\end{proof}
\section{Diameter bounds}
In this section,  we  consider diameter bounds. Next, we derive a criterion for the upper bound on $diam({\mathcal{P}}^*(A_n))$.
 \begin{corollary} 
 \label{le:whole}
If  ${\mathcal{P}}^*(A_n)$  is connected, then $  diam({\mathcal{P}}^*(A_n))\leq 11$.
 \,\label{le:19}
 \end{corollary}
 \begin{proof}
  Let $ x, y\in A_n $.   If $x $ and $ y $   are elements of prime orders, then  $d(x,y)\leq 10$,  by Lemmas ~\ref{le:3},  ~\ref{le:8}. 
  Suppose $o(y)$  is not prime and   $q$ is the least prime factor of   $o(y)$, thus  $ y\thicksim y^{o(y)/q}$ and let $ y^{o(y)/q}=y_1 \cdots y_{m'}  $ be the factorization of  $y^{o(y)/q}$ into cycles.  Hence        $m'= 3$, $n-|S(y^{o(y)/q})|\geq 3$ or $m'\neq3$.
If   $o(x)$ is not   prime, then $ x\thicksim x^{o(x)/p}$ and let  $x^{o(x)/p}=x_1 \cdots x_m$ be the factorization of  $x^{o(x)/p}$ into cycles, where   $p$ is the least prime factor of   $o(x)$.   Hence   $m= 3$,  $n-|S(x^{o(x)/p})|\geq 3$ or $m \neq3$.  
By Lemmas ~\ref{le:3} and ~\ref{le:8},   $d(x^{o(x)/p},y^{o(y)/q})\leq 8$. Therefore  $ d(x, y)\leq 10 $.
If  $o(x) $ be prime,
 then  $d(x, y^{o(y)/q} )\leq 10 $, by Lemmas ~\ref{le:8} and ~\ref{le:3}. Therefore  $ d(x, y)\leq 11 $. Consequently  
  $diam({\mathcal{P}}^*(A_n))\leq 11$. 
 \end{proof} 
 We need several results, to prove a lower bound of the diameter. 
 \\
\begin{lemma}  (\cite{a5},  Lemma 3.1.)
 Let $G$ be a finite group and  ${\mathcal{P}}^*(G)$ is connected.  
  If $x$ and $ y$ are elements of prime orders and  $\langle x \rangle\neq \langle y \rangle $, then $d(x, y)=2t$ and  there are elements $x_0 =  x,    x_1, \cdots,  x_{2t} = y$ such that $o(x_{2i})$ is prime for $ i \in \{0,  1,  \cdots,  t\} $, $o(x_{2i+1})=o(x_{2i}) o(x_{2i+2})$ for $ i \in \{0,  1,  \cdots,  t-1 \}$, and $x_i$ is adjacent to $x_{i+1}$ for $ i \in \{0,  1,  \cdots,   2t-1\}$, where   $t$ is a positive integer. 
\,\label{le:0}
\end{lemma}
 
\begin{theorem} ({Chebyshev's theorem} (\cite{d04}, p. 124))
If  $n>1$, then there is always at least one prime $p$,   between $n$ and $2n$.
\end{theorem}
 
\begin{lemma} (\cite{a8}, Theorem 2.)
\label{le:whole}
  Let $\alpha=\alpha_1\cdots\alpha_m$  such that $\alpha_1, \cdots, \alpha_m  $  are cycles of length   $  t $. Then $\langle \alpha_i \rangle \cong \mathbb{Z}_t$,  $  (1\leq i\leq m)$ and $C_{S_{mt}}(\alpha)\cong (\mathbb{Z}_t)^m\rtimes  S_m$,  where
 $C_{S_{mt}}(\alpha)$ is  the   centralizer  of $\alpha$ in $S_{mt}$.
 \,\label{le:46}
 \end{lemma} 
 
  
\begin{theorem}
\label{le:whole}
$  diam({\mathcal{P}}^*(A_n))\geq 6$.
 \,\label{le:20}
\end{theorem} 
\begin{proof}   
 By Chebyshev's theorem, there is a prime $p$, such that $[n/2]< p< n$. 
 Suppose $x=(1,\cdots, p)$ and $y=(n, n-1, \cdots, n-p+1)$. By Lemma  ~\ref{le:0}, $ d(x,y)=2t $, where $t$ is a positive integer.
 Since $o(x)=o(y)=p$, $d(x,y)>2$. 
   If $ t=2 $, then there are $ z, z_1 $ and $ z_2 $  in $S_n$  such that $ x\thicksim z_2\thicksim z\thicksim z_1\thicksim y $,  $ o(z_1)= o(z_2)=pq $ and $ o(z)=q $, where $q$ is a prime.   So $z\in C_{S_n}(x)\cap C_{S_n}(y)$.
Since    $C_{S_n}(x)\cong \langle x \rangle\times S_{\{p+1,\cdots, n\}}$, $C_{S_n}(y)\cong \langle y \rangle\times S_{\{ 1,\cdots, n-p\}}$ and $q\nmid p$, we have  $z\in S_{\{ 1,\cdots, n-p\}}\cap S_{\{p+1,\cdots, n\}}=1$,  which is a contradiction.
   Therefore  $d(x,y)\geq 6$. The proof is completed.
 \end{proof} 
  
\begin{lemma}
\label{le:whole}
Let    $p'$ is the maximum prime factor of the $ n(n-1)(n-2)$ and  $ [\frac{n-2}{p' }]\geq 3p'+2 $. Then $  diam({\mathcal{P}}^*(A_n))\leq 8$.
  \,\label{le:23}
\end{lemma}
\begin{proof} 
  Let $ x, y\in A_n $.  
Suppose $o(x)$ and $o(y)$  are not   primes,  and let
$p$  and $q$ are at least prime factors of $o(x)$  and  $o(y)$, respectively.
Then $ x\thicksim x^{o(x)/p}  $ and $ y\thicksim y^{o(y)/q}  $.   
 Now we let $x^{o(x)/p}=\alpha  =\alpha_1 \cdots \alpha_m $ and $y^{o(y)/q}=\beta=\beta_1 \cdots \beta_{m'}$. Assume that $ k=n-|S(\alpha)| $,   $ k'=n-|S(\beta)| $  and $p\leq q$.
If $pq\leq 9$, then $d(\alpha, \beta)\leq6$,  by Lemma ~\ref{le:3}. Therefore  $ d(x, y)\leq 8 $. \\
For  $pq>9$, we consider four cases as follows:\\  
{\bf Case 1:} Let $ p\neq q $, $ p\geq  5$ and $ q \geq 5$. 
\begin{itemize}
\item[$(1.1)$]
 If  $ k \geq 3 $, $ k'\geq 3 $ and $n-|S(\alpha)\cup S(\beta)|\geq 3$,
there is the $3$-cycle $c $  at distance $2$ from both $\alpha$ and  $\beta$, where $S(c)\subset S^c (\alpha)\cap S^c (\beta) $. Therefore  $ d(\alpha, \beta)\leq 4 $.
\item[$(1.2)$]
If  $ k \geq 3 $, $ k'\geq 3 $ and $n-|S(\alpha)\cup S(\beta)|< 3$, then $ mp+m'q>p'(3p'+2)$, thus $m+m'>3p'+2\geq 17$, accordingly $m\geq 7$ or $m'\geq 7$.
    Let $m'\geq 7$. 
    Since $ k \geq 3 $, there is the $3$-cycle $c $   with $d(\alpha, c)=2$.
    Since $m'\geq 7$,    there are   $x$, $\beta_{j_1}$, $\beta_{j_2}$,  $\beta_{j_3}$ and $\beta_{j_4}$ such that $S(c)\cap S(\beta_{j_1}\cdots\beta_{j_4})=\emptyset$, $d(x, \beta)=2$   and $S(x)=S(\beta_{j_1} \cdots \beta_{j_4})$, where $o(x)=2$.
   By Lemma ~\ref{le:1},   $ d(c, x)=2$.  Thus $ d(\alpha, \beta)\leq 6 $.
\item[$(1.3)$]
 If $k\geq3$ and $k'<3$, then $m'>7$. 
 Since $ k \geq 3 $, there is  a $3$-cycle $c $   with $d(\alpha, c)=2$.
Since $m'\geq 7$,  there is $y$ in $A_n$    such that   $d(\beta,y)=2$  and $S(y)\cap S(c)=\emptyset$, where $o(y)=2$. By Lemma ~\ref{le:1},  $d(c, y)=2$. Therefore  $ d(\alpha, \beta)\leq 6 $.
\item[$(1.4)$] If $k<3$ and $k'\geq 3$, then  the proof is similar to   $(1.3)$.
\item[$(1.5)$]
 If $k<3$ and $k'<3$, then  there are $\alpha_{i_1}$, $\alpha_{i_2}$, $\alpha_{i_3}$, $\beta_{j_1}$ and $\beta_{j_2}$ such that  $S(\alpha_{i_1})\cap S(\beta_{j_1})\neq\emptyset$, $|S(\alpha_{i_2})\cap S(\beta_{j_1}\beta_{j_2})|\geq 2$ and $S(\alpha_{i_3})\cap S(\beta_{j_2})\neq\emptyset$.
Also, there is $f$ in $A_n$  such that $d(f,\alpha)=2$ and $S(f)= S(\alpha_{i_1}\alpha_{i_2}\alpha_{i_3})$, where $o(f)=3$.  
 Hence  
     $S(f)\subset S(\beta_{j_1}\cdots \beta_{j_t})\cup S^c(\beta)$ and $t\leq 3p-2$. Since $m'\geq 3p+2$,  $m'-t\geq4$, thus there is $g $  in $A_n$  such that $d(\beta, g)=2$, 
       $S(g )=S(\beta_{j_{t+1}}\cdots\beta_{j_{t+4}})$, where $o(g)=2$.
Since $S(f)\cap S(g)=\emptyset$,     $d(f, g)=2$, by Lemma ~\ref{le:1}.   Therefore  $ d(\alpha, \beta)\leq 6 $. Consequently $d(x,y)\leq 8$.
\end{itemize}
The proof of the following cases is similar to  the previous case.\\  
{\bf Case 2:} Let $ p=q $ and $ p\geq 5 $. \\
{\bf Case 3:} Let $ q\geq 5 $ and $ p=3 $.\\
 {\bf Case 4:} Let $ q\geq 5 $  and $ p=2 $.  \\
 Suppose   $o(x)$ and $o(y)$  are prime.  
 As proved above  $d(x, y)\leq 6$. 
  Let $o(x)$ is a prime and $o(y)$  is not a prime. Then $y\sim y^{o(y)/q} $ and let  $y^{o(y)/q}=\beta=\beta_1 \cdots \beta_{m'}$, where $q$ is the least prime factor    $o(y)$. 
 As proved above $d(x, \beta)\leq 6$. Consequently $d(x,y)\leq 8$.
 The proof is completed.
\end{proof}
 \begin{example}
For $n= $  
$2025, 2432, 3250, 5292, 7106, 7569, 9802$, $ 10621$, $10880$, $10881$, $11286$, $11440$, $11662$, $13312$, $13456$, $13690$ and  $14337$,    
$  diam({\mathcal{P}}^*(A_n))\leq 8$.
 \end{example}
   
\begin{theorem}
Let   ${\mathcal{P}}^*(A_n)$  is connected and $p$ is the maximum prime factor of the $ n(n-1)(n-2)$.  Then   
\begin{itemize}
\item[$ (1) $]  If   $ [\frac{n-2}{p }]\geq 3p+2 $, then $6\leq diam({\mathcal{P}}^*(A_n))\leq 8$.
\item[$(2)$]  
    Otherwise, $6\leq diam({\mathcal{P}}^*(A_n))\leq 11$.
\end{itemize}
\end{theorem}
\begin{proof} 
By Theorem ~\ref{le:20}, Corollary ~\ref{le:19},  and Lemma  ~\ref{le:23},
 the proof is completed. \\
  \end{proof}
 
%
\hspace*{-7mm}
{\bf Acknowledgements}\\
We would like to   show our  gratitude to the referee for their careful review of our manuscript and some helpful suggestions.


\end{document}